\documentclass[10pt,leqno]{amsart}
\usepackage{amssymb,amsthm}
\usepackage{amsmath}
\usepackage{setspace}
\usepackage{dcpic,pictexwd}
\usepackage[all]{xy}
\usepackage[pdftex]{graphicx}
\usepackage{mathrsfs}
\usepackage[pdftex]{hyperref}
\usepackage{bbm}
\usepackage{subcaption}
\usepackage{tikz}
\usetikzlibrary{cd}

\theoremstyle{definition}
 \newtheorem{definition}{Definition}[section]

\theoremstyle{plain}

\theoremstyle{plain}
 \newtheorem{theorem}[definition]{Theorem}
 
\theoremstyle{definition}

\theoremstyle{plain}
 \newtheorem{lemma}[definition]{Lemma}

\theoremstyle{plain}

\theoremstyle{remark}
 \newtheorem{remark}[definition]{Remark}

\theoremstyle{definition}

\theoremstyle{plain}

\newcommand{\Ext}{\mathrm{Ext}}
\newcommand{\End}{\mathrm{End}}

\newcommand{\Hom}{\mathrm{Hom}}

\newcommand{\Ca}{\mathcal{C}}

\newcommand{\Fun}{\mathrm{F}}

\newcommand{\Def}{\mathrm{Def}}
\newcommand{\Sets}{\mathrm{Sets}}

\newcommand{\Ob}{\mathrm{Ob}}

\newcommand{\A}{\Lambda}
\newcommand{\Ab}{\mathscr{A}}

\newcommand{\m}{\mathfrak{m}}
\renewcommand{\k}{\Bbbk}

\newcommand{\invlim}{\varprojlim}

\newcommand{\V}{\widehat{V}}

\newcommand{\Ar}{\widehat{\Lambda}}

\hypersetup{
    bookmarks=true,         
    unicode=false,          
    pdftoolbar=true,        
    pdfmenubar=true,        
    pdffitwindow=false,     
    pdfstartview={FitH},    
    pdfnewwindow=true,      
    colorlinks=true,       
    linkcolor=red,          
    citecolor=blue,        
    filecolor=magenta,      
    urlcolor=blue           
}

\title[On universal deformation rings for objects of Ext-finite categories of modules]{On weak universal deformation rings for objects of Ext-finite categories of modules} 
\thanks{This research was partially supported by CODI (Universidad de Antioquia, UdeA), MINCIENCIAS (Convocatoria de Doctorado Nacional para profesores de IES 2021, Número 909), and the Office of Academic Affairs at the Valdosta State University.}

\author[L\'opez-Garc\'{\i}a]{Diego L\'opez-Garc\'{\i}a}
\address{Instituto de Matem\'aticas, Universidad de Antioquia, Medell{\'\i}n, Antioquia, Colombia}
\email{harvin.lopez@udea.edu.co}
\author[Rizzo]{Pedro Rizzo}
\address{Instituto de Matem\'aticas, Universidad de Antioquia, Medell{\'\i}n, Antioquia, Colombia}
\email{pedro.hernandez@udea.edu.co (Corresponding author)}
\author[V\'elez-Marulanda]{Jos\'e A. V\'elez-Marulanda}
\address{Department of Applied Mathematics \& Physics, Valdosta State University, Valdosta, GA,  United States of America}
\email{javelezmarulanda@valdosta.edu}
\address{Facultad de Matem\'aticas e Ingenier\'{\i}as, Fundaci\'on Universitaria Konrad Lorenz, Bogot\'a D.C.,  Colombia}
\email{josea.velezm@konradlorenz.edu.co}


\begin{document}
\renewcommand{\labelenumi}{\textup{(\roman{enumi})}}
\renewcommand{\labelenumii}{\textup{(\roman{enumi}.\alph{enumii})}}
\numberwithin{equation}{section}

\begin{abstract}
Let $\A$ be a $\k$-algebra where $\k$ a field of arbitrary characteristic, and let $\mathscr{A}_\k$ be a full subcategory of $\A$-Mod, the abelian category of left $\A$-modules. In particular, $\mathscr{A}_\k$ is a $\k$-category, i.e. the set of morphisms between objects in $\mathscr{A}_\k$ is a vector space over $\k$ and the composition of morphisms is $\k$-bilinear.  Following M. Kleiner and I. Reiten, $\mathscr{A}_\k$ is {\it Hom-finite} if the hom-space between any two objects in $\mathscr{A}_\k$  is finite-dimensional over $\k$. We further say that $\mathscr{A}_\k$ is {\it Ext-finite} if $\dim_\k\Ext^i_\A(X,Y)<\infty$ for all objects $X$ and $Y$ in $\mathscr{A}_\k$.  Let $V$ be an object in $\mathscr{A}_\k$. In this note we prove that if $\End_\A(V)$ is isomorphic to $\k$, then $V$ has a universal deformation ring $R(\A,V)$, which is a local complete Noetherian commutative $\k$-algebra whose residue field is also isomorphic to $\k$. We use this result to prove that if $\A$ is a local two-point infinite dimensional gentle $\k$-algebra (in the sense of V. Bekkert et al), then $R(\A,V)$ is isomorphic either to $\k$, to $\k[\![t]\!]/(t^2)$ or to $\k[\![t]\!]$. 
\end{abstract}
\maketitle

\section{Introduction}\label{sec1}

Let $\k$ be a field of arbitrary characteristic. Following \cite[\S 1]{kleiner-reiten}, an additive category $\Ab$ is a {\it $\k$-category} if for all objects $X$ and $Y$ in $\Ab$, the set of morphisms $\Hom_\Ab(X,Y)$ is a $\k$-vector space and the composition of morphims is $\k$-bilinear. If $\Ab$ is a $\k$-category, then $\Ab$ is {\it Hom-finite} if $\dim_\k\Hom_\Ab(X,Y)<\infty$. If $\Ab$ is further hom-finite, then we say that $\Ab$ is {\it Ext-finite} provided that $\dim_\k\Ext_\Ab^1(X,Y)<\infty$ for all objects $X$ and $Y$ in $\Ab$. Examples of Hom-finite and Ext-finite $\k$-categories is the abelian category of modules of finite-length over {\it admissible} algebras (as in \cite[\S 6]{krause4}). For a systematic study of a particular class of Hom-finite and Ext-finite $\k$-categories and their connection to specific classification problems, see \cite[pp. 857-859]{eno2018}. On the other hand, we denote by $\widehat{\Ca}$ the category of all complete local commutative Noetherian $\k$-algebras with residue field $\k$. In particular, the morphisms in $\widehat{\Ca}$ are continuous $\k$-algebra homomorphisms that induce the identity map on $\k$. 

Let $\A$ be a (not necessarily finite dimensional) $\k$-algebra, and let $\Ab_\k$ be a Hom-finite and Ext-finite full subcategory of $\A$-Mod, the abelian category of left $\A$-modules. Let $R$ be a fixed object in $\widehat{\Ca}$. We denote by $R\A$ the tensor product of $\k$-algebras $R\otimes_\k\A$. Let $V$ be a left $\A$-module in $\Ab_\k$. We say that $(M, \phi)$ is a {\it lift} of $V$ over $R$ if $M$ is  a finitely generated left $R\A$-module $M$ that  is free over $R$ and $\phi: \k\otimes_RM\to V$ is an isomorphism of left $\A$-modules. Two lifts $(M,\phi)$ and $(M', \phi')$ of $V$ over $R$ are said to be {\it isomorphic} if there exists an isomorphism of left $R\A$-modules $f: M\to M'$ such that $\phi = \phi'\circ (\mathrm{id}_\k\otimes f)$. A {\it deformation} of $V$ over $R$ is an isomorphism class of lifts of $V$ over $R$. We denote by $\Def_\A(V,R)$ the set of all deformations of $V$ over $R$. The {\it deformation functor} of $V$ is the covariant functor $\widehat{\Fun}_V: \widehat{\Ca}\to \Sets$ that sends $R\in \Ob(\widehat{\Ca})$ to $\Def_\A(V,R)$, and for each morphism $\theta: R\to R'$ in $\widehat{\Ca}$, $\widehat{\Fun}_V(\theta) : \Def_\A(V,R) \to \Def_\A(V, R')$ is the morphism that sends $[M, \phi]$ to $[R'\otimes_{R,\theta}M, \phi_\theta]$, where $\phi_\theta$ is the composition of the natural isomorphism $\k\otimes_{R'}(R'\otimes_{R,\theta}M$ with $\phi$.  The first main goal of this article is to prove the following result, which extends \cite[Prop. 2.1]{blehervelez} to Ext-finite categories of finite dimensional modules over associative $\k$-algebras. 

\begin{theorem}\label{thm1}
Let $\A$ be a (not necessarily finite dimensional) $\k$-algebra, and $\Ab_\k$ be a Hom-finite and Ext-finite full subcategory of $\A\textup{-Mod}$, and let $V$ be an object in $\Ab_\k$. If $\End_\A(V)=\k$, then $\widehat{\Fun}_V$ is represented by an object $R(\A,V)$ in $\widehat{\Ca}$.
\end{theorem}

The object $R(\A,V)$ in Theorem \ref{thm1} is unique up to canonical isomorphism and is called the {\it universal deformation ring} of $V$.

Recall that $\A$ is said to be a {\it two-point} algebra if $\A/\mathrm{rad}\,\A\cong \k^2$. Let $\k Q/I$ be a basic two-point $\k$-algebra and let $\A$ be its completion. Assume that $\A$ has infinite dimension over $\k$ and that it is also gentle (in the sense of \cite[Def. 2.5]{bekkert-drozd-futorny}). Then it follows from \cite[Prop. 2.6 (2)]{bekkert-drozd-futorny} that $\A$ is one of the algebras from Table \ref{table1}. Our second result is the following, which extends \cite[Thm. 5.2]{bekkert-giraldo-velez} for infinite dimensional gentle algebras. 

\begin{theorem}\label{thm2}
Let $\A$ be one of the algebras in Table \ref{table1}. If $V$ is a left $\A$-module such that $\End_\A(V)\cong \k$, then $R(\A, V)$ is isomorphic either to $\k$, to $\k[\![t]\!]/(t^2)$, or to $\k[\![t]\!]$.  
\end{theorem}

This work is partially motivated by \cite{fgrv}, where the authors develop a deformation theory for finite-dimensional modules over the repetitive algebra and explore its geometric applications. The repetitive algebra, which is infinite-dimensional, plays a crucial role in the classification of representation type for self-injective algebras (see, e.g., \cite{as3} and references therein). Furthermore, by Happel's celebrated result in \cite{happel}, it is fundamental in studying the derived category of bounded complexes of finitely generated modules over finite-dimensional algebras. In particular, the category of finite-dimensional modules over the repetitive algebra exemplifies a hom-finite and ext-finite category.

This article constitutes the doctoral dissertation of the first author under the supervision of the other two authors.  

\begin{table}
\centering
\begin{tabular}{|c|c|}\hline
\begin{tabular}{c}
$
Q' =\xymatrix@1@=20pt{
\underset{1}{\bullet}\ar@/^/[rr]^a&&\underset{2}{\bullet}\ar@/^/[ll]^b
}$
\\
$I = 0$
\end{tabular}
&
\begin{tabular}{c}
$
Q'' =\xymatrix@1@=20pt{
\underset{1}{\bullet}\ar@/^/[rr]^b\ar@/^2pc/[rr]^c&&\underset{2}{\bullet}\ar@/^/[ll]^c
}$
\\
$I = \langle ac,ca\rangle$
\end{tabular}\\\hline
\begin{tabular}{c}
$
Q''' =\xymatrix@1@=20pt{
\underset{1}{\bullet}\ar@/^/[rr]^b\ar@/^2pc/[rr]^a&&\underset{2}{\bullet}\ar@/^/[ll]^c\ar@/^2pc/[ll]^d
}$
\\
$I = \langle ca, db, ac, bd\rangle$\\
$I = \langle ca, db, bc, ad\rangle$
\end{tabular}
&
\begin{tabular}{c}
$
Q^{(iv)} =\xymatrix@1@=20pt{
\underset{1}{\bullet}\ar@(ul,dl)_a\ar[rr]^b&&\underset{2}{\bullet}
}$
\\\\
$I = \langle ba \rangle$
\end{tabular}\\\hline
\begin{tabular}{c}
$
Q^{(v)} =\xymatrix@1@=20pt{
\underset{1}{\bullet}\ar[rr]^b&&\underset{2}{\bullet}\ar@(ur,dr)^a
}$
\\
$I = \langle ab \rangle$
\end{tabular}
&
\begin{tabular}{c}
$
Q^{(vi)} =\xymatrix@1@=20pt{
\underset{1}{\bullet}\ar@(ul,dl)_a\ar@/^/[rr]^b&&\underset{2}{\bullet}\ar@/^/[ll]^c
}$
\\
$I = \langle a^2, bc \rangle$\\
$I = \langle ba, ac \rangle$\\
$I = \langle ba, ac, cb \rangle$
\end{tabular}\\\hline
\begin{tabular}{c}
$
Q^{(vii)} =\xymatrix@1@=20pt{
\underset{1}{\bullet}\ar@(ul,dl)_a\ar[rr]^c&&\underset{2}{\bullet}\ar@(ur,dr)^b
}$
\\
$I = \langle a^2, bc \rangle$\\
$I = \langle ca, b^2 \rangle$\\
$I = \langle ca, bc \rangle$
\end{tabular}
&
\begin{tabular}{c}
$
Q^{(viii)} =\xymatrix@1@=20pt{
\underset{1}{\bullet}\ar@(ul,dl)_a\ar@/^/[rr]^c&&\underset{2}{\bullet}\ar@/^/[ll]^d\ar@(ur,dr)^b
}$
\\
$I = \langle a^2, b^2, dc, cd \rangle$\\
$I = \langle a^2, db, bc, cd \rangle$\\
$I = \langle ca, db, bc, ad \rangle$
\end{tabular}\\\hline
\end{tabular}
\caption{Two-point infinite dimensional gentle algebras}\label{table1}
\end{table}
	
\section{Proof of Theorem \ref{thm1}}\label{sec2}

Through this section we assume the assumptions in the hypothesis of Theorem \ref{thm1}.   We denote by $\Fun_V$ the restriction of $\widehat{\Fun}_V$ to the full subcategory of Artinian objects in $\widehat{\Ca}$. Let $\k[\epsilon]$, with $\epsilon^2=0$, denote the ring of dual numbers over $\k$. The {\it tangent space} of $\widehat{\Fun}_V$ is defined to be the set $t_{V}=\Fun_V(\k[\epsilon])$. By \cite[Lemma 2.10]{sch}, $t_{V}$ has an structure of $\k$-vector space. We need to recall the following definition from \cite[Def. 1.2]{sch}. Following \cite[\S 14]{mazur}, we say that a functor $H: \widehat{\Ca}\to \Sets$ is {\it continuous} provided that for all objects $R$ in $\widehat{\Ca}$, 
\begin{equation}\label{cont}
H(R)=\invlim_iH(R/\mathfrak{m}_R^i),
\end{equation}
where $\mathfrak{m}_R$ denotes the unique maximal ideal of $R$. In order to prove Theorem \ref{thm1}, we need to prove the following more detailed version of it. 

\begin{theorem}\label{thm3}
The functor $\Fun_V$ has a pro-representable hull $R(\A,V)$ in $\widehat{\Ca}$ (as in \cite[Def. 2.7]{sch}), and the functor $\widehat{\Fun}_V$ is continuos. Moreover, there is a $\k$-vector space isomorphism $t_V\cong \Ext^1_\A(V,V)$. If $\End_\A(V)=\k$, then $\widehat{\Fun}_V$ is represented by $R(\A,V)$.
\end{theorem}

In order to prove that $\Fun_V$ has a pro-representable hull $R(\A,V)$, we need to check that $\Fun_V$ satisfies Schlessinger's criteria (H$_1$)-(H$_3$). In order to prove that $\Fun_V$ is representable, we need to check that  if $\End_\A(V)=\k$, then $\Fun_V$ also further satisfies (H$_4$) (see \cite[Thm. 2.11]{sch}).  More precisely, for all pullback diagrams of Artinian objects in $\widehat{\Ca}$, 

\begin{equation}\label{pullbackart}
\xymatrix@=20pt{
&R'''\ar[dl]_{\pi'}\ar[dr]^{\pi''}&\\
R'\ar[dr]_{\theta'}&&R''\ar[dl]^{\theta''}\\
&R&
}
\end{equation}
the induced map
\begin{equation}\label{thetapullback}
\Theta: \Fun_V(R''')\to \Fun_V(R') \times_{\Fun_V(R)}\Fun_V(R'')
\end{equation}
satisfies the following properties:
\begin{enumerate}
\item[(H$_1$)] $\Theta$ is surjective whenever $\theta'':R''\to R$ is a {\it small extension} in the sense of \cite[Def.1.2]{sch}, i.e. the kernel of $\theta''$ is a non-zero principal ideal $tR''$ that is annihilated by the unique maximal ideal $\m_R$ of $R$. ;
\item[(H$_2$)] $\Theta$ is a bijective when $R=\k$ and $R''=\k[\epsilon]$, where $\k[\epsilon]$ denotes the ring of dual numbers over $\k$ with $\epsilon^2=0$;
\item[(H$_3$)] the tangent space $t_{V}$ is finite dimensional over $\k$.
\end{enumerate}
Furthermore, if $\End_{\A}(V)=\k$, then the following property is also satisfied: 
\begin{enumerate}
\item[(H$_4$)] $\Theta$ is a bijection when $R'=R''$ and $\theta':R'\to R$ is a small extension. 
\end{enumerate}

\begin{lemma}\label{lemma3.12}
Schlessinger's criterion \textup{(H$_1$)} is satisfied by $\Fun_{V}$.
\end{lemma}

\begin{proof}
Assume that $\theta''$ in (\ref{pullbackart}) is a small extension Let $M'$ and $M''$ be lifts of $V$ over $R'$ and $R''$, respectively, and such that there exists an isomorphism of $R\A$-modules 
\begin{equation*}
\tau_R: R\otimes_{R',\theta'}M'\to R\otimes_{R'',\theta''}M''.
\end{equation*}
Let $M'''= M'\times_{R\otimes_{R',\theta'}M'}M''$. Since $M'$ and $M''$ are finitely generated as an $R'\A$-module and as an $R''\A$-module, respectively,  it follows that $M'''$ is as finitely generated as an $R'''\A$-module. Since $\theta''$ is surjective, and $M'$  and $M''$ are free over $R'$ and $R''$, respectively, it follows by \cite[Lemma 3.4]{sch} that $M'''$ is free over $R'''$. Moreover,  we also have that $R\otimes_{R'''}M'''\cong R\otimes_{R',\theta'}M'$ and thus we can have an isomorphism of left $\A$-modules $\phi''':\k\otimes_{R'''}M'''\to V$ as the composition
\begin{equation*}
\k\otimes_{R'''}M'''\cong \k\otimes_{R'}M'\xrightarrow{\phi'} V.
\end{equation*} 
Therefore, $M'''$ defines a lift of $V$ over $R'''$ such that 
\begin{align*}
\Fun_{V}(\pi')([M''', \phi''']) = [M', \phi'] \text{ and } \Fun_{V}(\pi'')([M''',\phi''']) = [M'', \phi''].
\end{align*}
Thus $\Theta$  in (\ref{thetapullback}) is surjective, which proves that (H$_1$) is satisfied by $\Fun_{V}$.  
\end{proof}

The following result can be proved by using the ideas in that of \cite[Prop. 4.3]{bleher14}. However, we decided to include a proof for the convenience of the reader.  

\begin{lemma}\label{lemma3.10}
Let $R$ be a fixed Artinian object in $\widehat{\Ca}$. If $M$ is a weak lift of $V$ over $R$, then the ring homomorphism $\sigma_{M}:R\to\End_{R\A}(M)$ coming from the action of $R$ on $M$ via scalar multiplication is an isomorphism of $R$-modules.
\end{lemma}

\begin{proof}
Let $R_0$ be an Artinian object in $\widehat{\Ca}$ such that  $\theta: R\to R_0$ is a small extension, and for all weak lifts $M_0$ of $V$ over $R_0$, the ring homomorphism $\sigma_{M_0}: R_0\to \End_{R_0\A}(M_0)$ coming from the action of $R_0$ on $M_0$ via scalar multiplication is an isomorphism of $R_0$-modules. Let $f\in \End_{R\A}(M)$. Then $\theta$ induces $f_0\in \End_{R_0\A}(M_0)$, where $M_0=R_0\otimes_{R,\theta}M$. By assumption, there exists $r_0\in R_0$ such that $f_0=\mu_{r_0}$, where $\mu_{r_0}$ denotes multiplication by $r_0$.  Let $r\in R$ such that $\theta(r)=r_0$, and let $\mu_r:M\to M$ be the scalar multiplication morphism by $r$.  Thus $\mathrm{id}_{R_0}\otimes \mu_r=\mu_{r_0}$. Let $\beta$ be the endomorphism of $M$ given by  $f-\mu_r$. It follows that $\mathrm{id}_{R_0}\otimes \beta=0$. Since $\theta$ is a small extension, we have that there is $t\in R$ such that $\ker \theta = t R$. This implies that $\mathrm{im}\, \beta\subseteq tM$ and thus $\beta \in \Hom_{R\A}(M, tM)$. Since $tR\cong \k$, it follows that $tM\cong \k\otimes_RM$ and there is an isomorphism $\Hom_{R\A}(M,tM)\cong\Hom_{\A}(\k\otimes_RM,\k\otimes_RM)\cong\End_{R\A}(V)$. Let $\beta_0$ be the image of $\beta$ under this isomorphism. Then there exist $\lambda_0\in \k$ such that $\beta_0=\mu_{\lambda_0}$, where $\mu_{\lambda_0}:V\longrightarrow V$ is multiplication by $\lambda_0$.  Thus $\beta=\mu_{t\lambda}$, where $\lambda\in R$ is such that $t\lambda$ gives $\lambda_0$ under the isomorphism $tR\cong \k$, and $\mu_{t\lambda}: M\to tM$ is multiplication by $t\lambda$. This implies that  $f= \mu_{t\lambda+r}$ in $\End_{R\A}(M)$ with $t\lambda + r\in R$ and thus $\sigma_{M}$ is surjective. Since $M$ is assumed to be free over $R$, it follows that $\sigma_M$ is also injective. This finishes the proof of Lemma \ref{lemma3.10}.  
\end{proof}

\begin{lemma}\label{lemma3.13}
Consider the pullback diagram of Artinian rings in $\widehat{\Ca}$ and let $\Theta$ be as in (\ref{thetapullback}). If either $R = \k$ or $\End_{\A}(V) = \k$, then $\Theta$ is injective. 
\end{lemma}
\begin{proof}
Let $(M', \phi')$ and $(M'', \phi'')$ be lifts of $V$ over $R'''$ such that there are isomorphisms of $R'\A$-modules and $R''\A$-modules, respectively
\begin{align*}
\tau_{R'} : R'\otimes_{R''', \pi'}M'\to R'\otimes_{R''', \pi'}M'' \text{ and }  \tau_{R''}: R''\otimes_{R''', \pi''}M' \to R''\otimes_{R''', \pi''}M'',
\end{align*}
such that $\phi' = \phi''\circ (\mathrm{id}_\k\otimes_{R'} \tau_{R'})$ and $\phi' = \phi''\circ (\mathrm{id}_\k\otimes_{R''} \tau_{R''})$ 
Let $\varphi_{R}  = (\mathrm{id}_R\otimes \tau_{R'})^{-1}(\mathrm{id}_R\otimes \tau_{R''})\in \End_{R\A}(R\otimes_{\theta''\circ \pi''}M'')$. If $R = \k$, then $\varphi_{\k} = \mathrm{id}_{V}$. Assume next that $\End_\A(V) = \k$. Since $R\otimes_{\theta''\circ \pi''}M''$ is a lift of $V$ over $R$, it follows by Lemma \ref{lemma3.10} that there exists an scalar $r\in R^\ast$ such that $\varphi_R= \mu_r$. Let $r''\in R''$ such that $\theta''(r'')= r$ and replace $\tau_{R''}$ for $\mu_{r''}^{-1}\tau_{R''}$, which gives that $\varphi_R = \mathrm{id}_{R\otimes_{\theta''\circ \pi''}M'}$. Thus in both situations, we can assume that $\varphi_R$ is just the identity on $R\otimes_{\theta''\circ \pi''}M'$. In particular $\mathrm{id}_R\otimes\tau_{R'} = \mathrm{id}_{R''}\otimes\tau_{R''}$.  By using the isomorphisms of $R''\A$-modules

\begin{align}
M' &= (R'\otimes_{R''', \pi'}M')\times_{R\otimes_{\theta'\circ \pi'}M'}(R''\otimes_{R''', \pi''}M'),\\\notag
M'' &= (R'\otimes_{R''', \pi'}M'')\times_{R\otimes_{\theta'\circ \pi'}M''}(R''\otimes_{R''', \pi''}M''),
\end{align}
\noindent
we obtain a well-defined morphism of $R'''\A$-modules $\tau_{R'''}: M'\to M''$ given by $\tau_{R'''} = (\tau_{R'}, \tau_{R''})$, which is an isomorphism that satisfies $\varphi' = \varphi''\circ (\mathrm{id}_\k\otimes \tau_{R'''}) $. This proves that $\Theta$ is injective and finishes the proof of Lemma \ref{lemma3.13}.  
\end{proof}

As a direct consequence of Lemma \ref{lemma3.13} we obtain that $\Fun_V$ always satisfies Schlessinger's criteria \textup{(H$_2$)} as well as \textup{(H$_4$)} provided that $\End_\A(V) = \k$.

\begin{lemma}\label{lemma3.14}
There exists an isomorphism of $\k$-vector spaces $t_{V}\cong \Ext_{\A}^1(V,V)$. 
\end{lemma}

\begin{proof}
As in the proof of \cite[Prop. 2.1]{blehervelez}, the isomorphism $t_{V}\cong \Ext_{\A}^1(V,V)$ is established in a similar manner to that in \cite[\S 22]{mazur}. Namely, for a given lift $(M, \phi)$ of $V$ over the ring of dual numbers $\k[\epsilon]$, we have $\Ar$-module isomorphisms $\epsilon M\cong V$ and $M/\epsilon M\cong V$. Therefore, we obtain a short exact sequence of left $\A$-modules $\mathcal{E}_{M}: 0\to V\to M\to V\to 0$.
Thus we have a well defined $\k$-linear map $s:t_{V}\to \Ext_{\A}^1(V,V)$ defined as $s([M,\phi])= \mathcal{E}_{M}$. Now, if $\mathcal{E}: 0\to V\xrightarrow{\alpha_1}V_1\xrightarrow{\alpha_2}V\to 0$
is an extension of $\A$-modules, then $V_1$ is naturally a $\k[\epsilon]\A$-module by letting $\epsilon\cdot v_1 = (\alpha_1\circ \alpha_2)(v_1)$ for all $v_1\in V_1$.   In this way, $V_1$ is free over $\k[\epsilon]$ and there is a left $\A$-module isomorphism $\psi: V_1/\epsilon V_1\to V$. Therefore, $(V_1,\psi)$ defines a lift of $V$ over $\k[\epsilon]$. Thus the $\k$-linear map $s': \Ext_{\A}^1(V,V)\to t_{V}$ defined as $s'(\mathcal{E})= [V_1, \psi]$ gives an inverse of $s$, which implies that $t_{V}\cong \Ext_{\A}^1(V,V)$ as $\k$-vector spaces. 
\end{proof}

It follows from Lemma \ref{lemma3.14} and the fact that $V$ is an object in the Ext-finite category $\Ab_\k$ that $\Fun_V$ always satisfies Schlessinger's criterion \textup{(H$_3$)}.

In order to finish the proof of Theorem \ref{thm3}, we next prove the continuity of the deformation functor $\widehat{\Fun}_{V}$ in the next Lemma \ref{lemma3.15}. As with the proof of Lemma \ref{lemma3.14}, the proof of Lemma \ref{lemma3.15} can be obtained by adapting the arguments in the proofs of \cite[Prop. 2.1]{blehervelez} and \cite[Prop. 2.4.4]{blehervelez2} to our context (see also \cite[\S 20, Prop. 1]{mazur}). We decided to include a proof for the convenience of the reader. 

\begin{lemma}\label{lemma3.15}
The functor $\widehat{\Fun}_{V}:\widehat{\Ca}\to \Sets$ is continuous.
\end{lemma}


\begin{proof}
For all objects $R$ in $\widehat{\Ca}$, we consider the natural map 
\begin{equation}\label{contfunt}
\Gamma: \widehat{\Fun}_{V}(R)\to \invlim_i\Fun_{V}(R_i)
\end{equation}
defined by $\Gamma([M, \phi])=\{[M_i,\phi_i]\}_{i=1}^\infty$, where for all $i\geq 1$, $R_i=R/\m_R^i$, $\pi_i:R\to R_i$ is the natural projection, and $[M_i,\phi_i]=\widehat{\Fun}_V(\pi_i)([M,\phi])$. We first prove that $\Gamma$ as in (\ref{contfunt}) is surjective. Let $\{[N_i,\psi_i]\}_{i=1}^\infty\in  \invlim_i\Fun_V(R_i)$. Then for each $i\geq 1$, there exists an isomorphism of $R_i\A$-modules
\begin{equation*}
\widehat{\alpha}_i:R_i\otimes_{R_{i+1}}N_{i+1}\to N_i
\end{equation*}
such that $\psi_i\circ (\mathrm{id}_\k\otimes \widehat{\alpha}_i)=\psi_{i+1}$. Let $(N'_1,\psi'_1)= (N_1,\psi_1)$. For each $i\geq 2$, the natural surjection $R_{i+1}\to R_i$ induces a surjective $R_{i+1}\A$-module homomorphism
\begin{equation*}
N_{i+1}\to R_i\otimes_{R_{i+1}}N_{i+1}.
\end{equation*}
\noindent
Hence, for all $i\geq 2$, define $(N'_i,\psi'_i)$ as follows: $N'_i = R_i\otimes_{R_i+1}N_{i+1}$ and $\psi'_i$ is the composition
\begin{equation*}
\k\otimes_{R_i}N'_i\cong \k\otimes_{R_{i+1}}N_{i+1}\xrightarrow{\psi_{i+1}}V.
\end{equation*}
Then for all $i\geq 1$, $(N'_i,\psi'_i)$ is a lift $V$ over $R_i$, which is  also isomorphic to $(N_i, \psi_i)$  and there exists surjective $R_i\Ar$-module homomorphisms 
\begin{equation*}
\beta^{i+1}_i: N'_{i+1}\to N'_i 
\end{equation*}
that induces a natural isomorphism $R_i\otimes_{R_{i+1}}N'_{i+1} = N'_i$ which preserves the $\Ar$-module structure. Moreover $\psi'_i$ is equal to the composition 
\begin{equation*}
\k\otimes_{R_i}N'_i=\k\otimes_{R_i}(R_i\otimes_{R_{i+1}}N'_{i+1})\cong \k\otimes_{R_{i+1}}N'_{i+1}\xrightarrow{\psi'_{i+1}}V.  
\end{equation*}
Therefore $\{[N'_i, \psi_i],\beta^{i+1}_i\}_{i=1}^\infty$ forms an inverse system of deformations of $\V$. If we let $N'=\invlim_iN'_i$  and $\psi=\invlim_i\psi'_i$, then we have that $(N', \psi')$ is a lift of $V$ over $R=\invlim_iR_i$ such that $\Gamma([N', \psi'])=\{[N_i, \psi_i]\}_{i=1}^\infty$. This proves that $\Gamma$ is surjective. In order to prove that $\Gamma$ is injective, assume that $[M, \phi]$ and $[M', \phi']$ are lifts of $V$ over $R$ such that $\Gamma([M, \phi])=\Gamma([M', \phi'])$. Then for all $i\geq 1$, there is an isomorphism of $R_i\A$-modules 
\begin{equation*}
\gamma_i:R_i\otimes_RM\to R_i\otimes_RM', 
\end{equation*}
such that $\phi'\circ (\mathrm{id}_\k\otimes_{R_i}\gamma_i)=\phi$. 
For each $i\geq 1$, define 
\begin{equation*}
\zeta_i=\gamma_i^{-1}\circ (\mathrm{id}_{R_i}\otimes_{R_{i+1}}\gamma_{i+1})-\mathrm{id}_{R_i\otimes_RM},
\end{equation*}
and set $\zeta_i^{(i)}=\zeta_i$. Observe that $\mathrm{id}_{\k}\otimes_{R_i}\zeta_i=0$. Since the $R_i$ are Artinian, for each $i\geq 2$, and $i\leq \ell< j$, we can lift the $R_\ell\A$-module homomorphism $\zeta_i^{(\ell)}\in \End_{R_\ell\A}(R_\ell\otimes_R M)$ to an $R_j\A$-module homomorphism $\zeta_i^{(j)}\in \End_{R_j\A}(R_j\otimes_R M)$ such that $\mathrm{id}_{R_\ell}\otimes_{R_j}\zeta_i^{(j)}=\zeta_i^{(\ell)}$. Moreover, since for all $2\leq i\leq \ell< j$, the morphism $\zeta_i^{(j)}$ is nilpotent, it follows that $\mathrm{id}_{R_j\otimes_RM} + \zeta_{i}^{(j)}$ is invertible.  Next let $f_1 = \gamma_1$ and $f_2 = \gamma_2$, and for all $j\geq 3$, let 
\begin{align*}
f_j=\gamma_j\circ (\mathrm{id}_{R_j\otimes_RM} + \zeta_{j-1}^{(j)})^{-1}\circ (\mathrm{id}_{R_j\otimes_RM} + \zeta_{j-2}^{(j)})^{-1}\circ \cdots \circ (\mathrm{id}_{R_j\otimes_RM} + \zeta_{2}^{(j)})^{-1}
\end{align*}
Note that $\mathrm{id}_{R_1}\otimes_{R_2}f_2=f_1$. On the other hand, if $j\geq 2$, then 
\begin{align*}
\mathrm{id}_{R_j}\otimes_{R_{j+1}}f_{j+1}&= (\mathrm{id}_{R_j}\otimes_{R_{j+1}}\gamma_{j+1})\circ  (\mathrm{id}_{R_j\otimes_RM} + \zeta_{j}^{(j)})^{-1}\circ (\mathrm{id}_{R_j\otimes_RM} + \zeta_{j-1}^{(j)})^{-1}\circ \cdots \circ (\mathrm{id}_{R_j\otimes_RM} + \zeta_{2}^{(j)})^{-1}\\
&= \gamma_j\circ (\mathrm{id}_{R_j\otimes_RM}+\zeta_j)\circ  (\mathrm{id}_{R_j\otimes_RM} + \zeta_j)^{-1}\circ (\mathrm{id}_{R_j\otimes_RM} + \zeta_{j-1}^{(j)})^{-1}\circ \cdots \circ (\mathrm{id}_{R_j\otimes_RM} + \zeta_{2}^{(j)})^{-1}\\
&=f_j.
\end{align*}
Moreover, for each $i\geq 1$, we have 
\begin{align*}
\phi'\circ (\mathrm{id}_\k\otimes_{R_i} f_i)&= \phi'\circ (\mathrm{id}_\k\otimes_{R_i} \gamma_i)\circ (\mathrm{id}_{\k\otimes_R M}+\mathrm{id}_\k\otimes_{R_{i-1}}\zeta_{i-1})^{-1}\circ \cdots \circ (\mathrm{id}_{\k\otimes_R M}+\mathrm{id}_\k\otimes_{R_2}\zeta_{2})^{-1}\\
&= \phi. 
\end{align*}
Let $f=\invlim_if_i$. Then $f:M\to M'$ is an isomorphism of $R\A$-modules such that $\phi'\circ (\mathrm{id}_\k\otimes_Rf)= \phi$. This proves that $\Gamma$ is injective. This finishes the proof of Lemma \ref{lemma3.15} and therefore the proof of Theorem \ref{thm3}.
\end{proof}

\begin{remark}
Assume that $(M,\phi)$ and $(M', \phi')$ are lifts of $V$ over $R$ such that there exists an isomorphism of $R\A$-modules $f:M\to M'$. Since $\End_\A(V)=\k$, we can re-define $f$ to $\overline{f}$ by using multiplication by an invertible scalar $\lambda \in R^\ast$ such that $\phi'\circ (\mathrm{id}_\k\otimes \overline{f}) = \phi$. Therefore $[M,\phi] = [M,\phi']$. In this situation we have that there exists a natural equivalence between the deformation functor $\widehat{\Fun}_V$ and the {\it weak deformation functor} $\widehat{\Fun}^w_V$ (as discussed in \cite{rizzo-velez}). Therefore the universal deformation ring $R(\A,V)$ and the weak universal deformation ring $R^w(\A,V)$ (as discussed in \cite{rizzo-velez}) are isomorphic in $\widehat{\Ca}$.   
\end{remark}

We need the following result from \cite[Thm. 1.1]{rizzo-velez}.

\begin{theorem}\label{thm8}
Let $\A$ be a (not necessarily finite dimensional) $\k$-algebra and $V$ be an indecomposable left $\A$-module with $\dim_\k V<\infty$ and $\dim_\k\Ext_\A^1(V,V)=1$. Assume that $V$ has a weak universal deformation ring $R^w(\A,V)$, and that there exists an ordered sequence of indecomposable finite dimensional left $\A$-modules (up to isomorphism) $\mathscr{L}_V= \{V_0,V_1,\ldots\}$ with $V_0 = V$ and such that for each $\ell \geq 1$, there exist a surjective morphism  $\epsilon_\ell: V_\ell\to V_{\ell-1}$, and an injective morphism $\iota_\ell: V_{\ell-1}\to V_\ell$ such that the composition $\sigma_\ell = \iota_\ell\circ \epsilon_\ell$ satisfies that $\ker \sigma_\ell = V_0$, $\mathrm{im}\,\sigma_\ell^\ell \cong V_0$, $\sigma_\ell^{\ell+1}=0$, and $\mathscr{L}_V$ is maximal in the sense that if there is another ordered sequence of indecomposable left $\A$-modules $\mathscr{L}'_V$ with these properties, then $\mathscr{L}'_V\subseteq \mathscr{L}_V$. 
\begin{enumerate}
\item If $\mathscr{L}_V$ is finite, and its last element, say $V_N$, satisfies  $\dim_\k\Hom_\A(V_N,V) =1$ and $\Ext_\A^1(V_N,V)=0$, then  $R^w(\A,V)\cong \k[\![t]\!]/(t^{N+1})$.
\item If $\mathscr{L}_V$ is infinite, then $R^w(\A,V)\cong \k[\![t]\!]$.
\end{enumerate}   
\end{theorem} 

\section{Proof of Theorem \ref{thm2}}

In what follows we consider the $\k$-algebra $\A_0 = \k Q^{(viii)}/\langle a^2, b^2, dc, cd\rangle$, where $Q^{(viii)}$ is as in Table \ref{table1}. The result in Theorem \ref{thm2} is verified analogously for the other cases. 
The radical series of the projective indecomposable $\A_0$-modules are as in Figure \ref{projec}.

\begin{figure}[ht]
\begin{align*}
P_1=\begindc{\commdiag}[100]
\obj(-1,-1)[v1]{$S_1$}
\obj(0,0)[v2]{$S_1$}
\obj(1,-1)[v6]{$S_2$}
\obj(2,-2)[v7]{$S_2$}
\obj(3,-3)[v7]{$S_1$}
\obj(4,-4)[v7]{$S_1$}
\obj(5,-5)[v7]{$S_2$}
\obj(6,-6)[v7]{$\ddots$}
\enddc
&&
P_2=\begindc{\commdiag}[100]
\obj(-1,-1)[v1]{$S_2$}
\obj(0,0)[v2]{$S_2$}
\obj(1,-1)[v6]{$S_1$}
\obj(2,-2)[v7]{$S_1$}
\obj(3,-3)[v7]{$S_2$}
\obj(4,-4)[v7]{$S_2$}
\obj(5,-5)[v7]{$S_1$}
\obj(6,-6)[v7]{$\ddots$}
\enddc
\end{align*}
\caption{The radical series of the projective indecomposable $\A_0$-modules corresponding to the vertices of $Q^{(viii)}$.}\label{projec}
\end{figure} 

The following result is an immediate consequence of the description of morphisms between string modules for string algebras provided by H. Krause in \cite{krause}. 

\begin{lemma}\label{lemma3.2}
Let $V$ be a finite dimensional $\A_0$-module. Then  $\End_{\A_0}(V)=\k$ if and only if $V$ is a simple $\A_0$-module or $V$ is a string $\A_0$-module $V= M[\beta]$, where $\beta = \{c,d,ca,ad,db,bc,bca,adb\}$.  
\end{lemma} 

\begin{remark}\label{rem3.3}
Assume that $V$ is a finite dimensional $\A_0$-module such that $\End_{\A_0}(V)=\k$. If $\dim_\k \Ext_{\A_0}^1(V,V)=n$, then it follows by using the same arguments as those in \cite[First paragraph on pg. 223]{bleher15} that $R(\A_0,V)$ is a quotient of the ring of formal power series $\k[\![t_1,\ldots,t_n]\!]$. In particular, if $\Ext_{\A_0}^1(V,V) =0$, then $R(\A_0, V)=\k$.
\end{remark}

\begin{lemma}\label{lemma3.3}
Let $V$ be finite dimensional $\A_0$-module with $\End_{\A_0}(V)=\k$. 
\begin{enumerate}
\item If $V= M[\beta]$ with $\beta\in \{c,d,ca,ad,db,bc\}$, then $\dim_\k\Ext^1_{\A_0}(V,V)=0$ and  $R(\A_0, V) = \k$. 
\item If $V$ is a simple $\A_0$-module, then $R(\A_0,V)=\k[\![t]\!]/(t^2)$.
\item If $V = M[\beta]$ with $\beta \in \{bca,adb\}$, then $R(\A_0,V)=\k[\![t]\!]$. 
\end{enumerate}
\end{lemma}

\begin{proof}

Assume that $V$ is as in the hypothesis of (i). It is straightforward to verify that $\Ext^1_{\A_0}(V,V) = 0$. By using Remark \ref{rem3.3}, it follows that $R(\A_0, V) = \k$. Assume next that $V$ is as in (ii). Without loss of generality assume that $V= S_1$. It follows that $\mathscr{L}_V=\{V_0, V_1\}$, where $V_0 = S_1$, $V_1$ is the string module $M[a]$, and $\mathscr{L}_V$ is as in Theorem \ref{thm8} (i). It follows that $R(\A_0,V)= \k[\![t]\!]/(t^2)$. Finally, assume that $V$ is as in the hypothesis of (iii). Without loss of generality, assume that $V$ is the string module $M[bca]$. Then $\mathscr{L}_V = \{V_0, V_1, V_2,\ldots\}$, where $V_0 = M[bca]$, for all $n\geq 1$, $V_n = M[bca(dbca)^n]$, and  $\mathscr{L}_V$ is as in Theorem \ref{thm8} (ii). It follows that $R(\A_0, V) = \k[\![t]\!]$.
\end{proof}

By using the same reasoning in the proof of Lemma \ref{lemma3.3} for the remainder algebras in Table \ref{table1}, we obtain the complete proof of Theorem \ref{thm2}.

\begin{remark}
To clarify the final claim, we propose an algorithmic approach based on Theorem \ref{thm8} and the proof technique of Lemma \ref{lemma3.3} to systematically address the remaining cases in Table \ref{table1}. Consider a string module $V=M[\beta]$ associated with the string $\beta$ of a two point, infinite dimensional gentle $\Bbbk$-algebra $\Lambda$ from Table \ref{table1}, where $\End_{\A}(V)=\k$. If $\dim_\k\Ext^1_{\A}(V,V)=0$, then $R(\A, V) = \k$. Now, if $\dim_\k\Ext^1_{\A}(V,V)=1$ and a \textit{connecting arrow} $c$ exists, i.e., either an arrow or its inverse such that $\beta c\beta$ forms a string module in $\Lambda$, two situations arise: 
\begin{enumerate}
\item[i)] If $(\beta c)^n\beta$ is a string in $\Lambda$ for all $n\geq 1$, an infinite ordered sequence of indecomposable finite dimensional left $\A$-modules $\mathscr{L}_V= \{V_0,V_1,\ldots\}$, with $V_0 = V$ and $V_n=M[(\beta c)^n \beta]$, satisfying the conditions of Theorem \ref{thm8} ii), is obtained.
\item[ii)] Otherwise, for a minimal integer $N\geq1$ with $(\beta c)^N\beta\neq 0$, a finite ordered sequence of indecomposable finite dimensional left $\A$-modules $\mathscr{L}_V= \{V_0,V_1,\ldots, V_N\}$, with $V_0 = V$ and $V_N=M[(\beta c)^N \beta]$, satisfying  the conditions in Theorem \ref{thm8} i), is produced. Actually, in our cases $N$ is reduced to be equal to 1.
\end{enumerate}
As described in Lemma \ref{lemma3.2}, every indecomposable module $V$ (with $\End_{\A}(V)=\k$) is either simple or isomorphic to a string module $M[\beta]$, where $\beta$ admits a connecting arrow. Consequently, this characterization, combined with the aforementioned case-by-case analysis, yields a complete proof of Theorem \ref{thm2}.
\end{remark}

\section{Ethical Approval}

Not applicable

\section{Funding}  
The first author was partially supported by MINCIENCIAS, Convocatoria de Doctorado Nacional para profesores de IES 2021, Número 909. The second author was partially supported by CODI (Universidad de Antioquia, UdeA), by project numbers 2022-52654 and 2023-62291. The third author was supported by the research group PROMETE-KONRAD (Project No. 5INV1232) in Facultad de Matem\'aticas e Ingenier\'{\i}as at the Fundaci\'on Universitaria Konrad Lorenz, Bogot\'a, Colombia, by the Faculty Scholarship of the Office of Academic Affairs at the Valdosta State University, GA, USA,  and by the research group \'Algebra U. de. A in the Instituto de Matem\'aticas  at the Universidad de Antioquia in Medell\'{\i}n, Colombia. 

\section{Availability of data and materials}  

Not applicable.


\end{document}